\documentclass[11pt]{amsart}

\usepackage[height=190mm,width=130mm]{geometry}
\numberwithin{equation}{section} 

\newtheorem{thm}{Theorem}[section]

\newtheorem{conj}[thm]{Conjecture}
\theoremstyle{definition}
\newtheorem{defn}[thm]{Definition}
\theoremstyle{remark}
\newtheorem{rem}[thm]{Remark}

\newcommand{\bea}{\begin{eqnarray}}
\newcommand{\eea}{\end{eqnarray}}
\newcommand{\ba}{\begin{array}}
\newcommand{\ea}{\end{array}}
\newcommand{\bc}{\begin{center}}
\newcommand{\ec}{\end{center}}
\newcommand{\be}{\begin{equation}}
\newcommand{\ee}{\end{equation}}

\def\d{\delta}

\def\l{\lambda}

\def\cf{{\mathcal F}}

\def\br{\mathbb{R}}
\def\a{\alpha}

\def\b{\beta}
\def\m{\mu}
\def\n{\nu}

\def\g{\gamma}

\begin{document}

\title[ORTHOGONALITY PRESERVING]
{ON VOLTERRA AND ORTHOGONALITY PRESERVING QUADRATIC STOCHAISTIC
OPERATORS}

\author{Farrukh Mukhamedov}
\address{Farrukh Mukhamedov\\
 Department of Computational \& Theoretical Sciences\\
Faculty of Science, International Islamic University Malaysia\\
P.O. Box, 141, 25710, Kuantan\\
Pahang, Malaysia} \email{{\tt far75m@yandex.ru} {\tt
farrukh\_m@iium.edu.my}}

\author{Muhammad Hafizuddin Bin Mohd Taha}
\address{Muhammad Hafizuddin Bin Mohd Taha\\
 Department of Computational \& Theoretical Sciences\\
Faculty of Science, International Islamic University Malaysia\\
P.O. Box, 141, 25710, Kuantan\\
Pahang, Malaysia}



\begin{abstract}
 A quadratic stochastic operator (in short QSO) is usually used to
present the time evolution of differing species in biology. Some
quadratic stochastic operators have been studied by Lotka and
Volterra. In the present paper, we first give a simple
characterization of Volterra QSO in terms of absolutely continuity
of discrete measures. Moreover, we provide its generalization in
continuous setting. Further, we introduce a notion of orthogonal
preserving QSO, and describe such kind of operators defined on two
dimensional simplex. It turns out that orthogonal preserving QSOs
are permutations of Volterra QSO. The associativity of genetic
algebras generated by orthogonal preserving QSO is studied too.
\end{abstract}

\maketitle

\section{Introduction}

The history of quadratic stochastic operators (QSO) can be traced
back to Bernstein's work \cite{B} where such kind of operators
appeared from the problems of population genetics (see also
\cite{Ly2}). Such kind of operators describe time evolution of
variety species in biology are represented by so-called
Lotka-Volterra(LV) systems \cite{L,V1,V2}.

A quadratic stochastic operator is usually used to present the time
evolution of species in biology, which arises as follows. Consider a
population consisting of $m$ species (or traits) $1,2,\cdots,m$. We
denote a set of all species (traits) by $I=\{1,2,\cdots,m\}$.  Let
$x^{(0)}=\left(x_1^{(0)},\cdots,x_m^{(0)}\right)$ be a probability
distribution of species at an initial state and $P_{ij,k}$ be a
probability that individuals in the $i^{th}$ and $j^{th}$ species
(traits) interbreed to produce an individual from $k^{th}$ species
(trait). Then a probability distribution
$x^{(1)}=\left(x_{1}^{(1)},\cdots,x_{m}^{(1)}\right)$ of the spices
(traits) in the first generation can be found as a total
probability,  i.e.,
\begin{equation*} x_k^{(1)}=\sum_{i,j=1}^m P_{ij,k} x_i^{(0)} x_j^{(0)}, \quad k=\overline{1,m}.
\end{equation*}
This means that the association $x^{(0)} \to x^{(1)}$ defines a
mapping $V$ called \textit{the evolution operator}. The population
evolves by starting from an arbitrary state $x^{(0)},$ then passing
to the state $x^{(1)}=V(x^{(0)})$ (the first generation), then to
the state
$x^{(2)}=V(x^{(1)})=V(V(x^{(0)}))=V^{(2)}\left(x^{(0)}\right)$ (the
second generation), and so on. Therefore, the evolution states of
the population system  described by the following discrete dynamical
system
$$x^{(0)}, \quad x^{(1)}=V\left(x^{(0)}\right), \quad x^{(2)}=V^{(2)}\left(x^{(0)}\right), \quad x^{(3)}=V^{(3)}\left(x^{(0)}\right) \cdots$$

In other words, a QSO describes a distribution of the next
generation if the distribution of the current generation was given.
The fascinating applications of QSO to population genetics were
given in \cite{Ly2}. Furthermore, the quadratic stochastic operator
was considered an important source of analysis for the study of
dynamical properties and modelings in various fields such as biology
\cite{HHJ,HS,May,MO,NSE}, physics \cite{PL,T}, economics and
mathematics \cite{G,Ly2,T,U,V}.

In \cite{11}, it was given along self-contained exposition of the
recent achievements and open problems in the theory of the QSO. The
main problem in the nonlinear operator theory is to study the
behavior of nonlinear operators. This problem was not fully finished
even in the class of QSO (the QSO is the simplest nonlinear
operator). The difficulty of the problem depends on the given cubic
matrix $(P_{ijk})_{i,j,k=1}^m$. An asymptotic behavior of the QSO
even on the small dimensional simplex is complicated \cite{MSQ,V,Z}.

In the present paper, we first give a simple characterization of
Volterra QSO (see \cite{G}) in terms of absolutely continuity of
discrete measures (see Section 3). Further, in section 4 we
introduce a notion of orthogonal preserving QSO, and describe such
kind of operators defined on two dimensional simplex. It turns out
that orthogonal preserving QSOs are permutations of Volterra QSO. In
section 5, we study associativity of genetic algebras generated by
orthogonal preserving QSO.

\section{Preliminaries}

An evolutionary operator of a free population is a (quadratic)
mapping of the simplex
\begin{equation}\label{1.2}
S^{m-1}=\{\mathbf{x}=(x_1,\ldots,x_m)\in \mathbb{R}^m|x_i\geq0, \ \
\sum_{i=1}^mx_i=1\}
\end{equation}
into itself of the form
\begin{equation}\label{1.3}
V:x_k^{\prime}=\sum_{i,j=1}^mP_{ij,k}x_ix_j, \ \ k=1,2,\ldots,m
\end{equation}
where $P_{ij,k}$ are coefficient of heredity and
\begin{equation}\label{1.4}
P_{ij,k}\geq0, \ \ P_{ij,k}=P_{ji,k}, \ \ \sum_{k=1}^{m}P_{ij,k}=1,
\ \ i,j,k=1,2,\ldots,m
\end{equation}
Such a mapping is called \textit{quadratic stochastic operator
(QSO)}.

Note that every element $\mathbf{x}\in S^{m-1}$ is a probability
distribution on $E=\{1,\ldots,m\}$. The population evolves starting
from an arbitrary initial state $\mathbf{x}\in S^{m-1}$ (probability
distribution on $E$) to the state
$\mathbf{x}^{\prime}=V(\mathbf{x})$ in the next generation, then to
the state
$\mathbf{x}^{\prime\prime}=V^2(\mathbf{x})=V(V(\mathbf{}x))$, and so
on.

For a given $\mathbf{x}^{(0)}\in S^{m-1}$, the trajectory
$$\{x^{(n)}\}, \ \ m=0,1,2,\ldots$$
of $\mathbf{x}^{(0)}$ under the action of QSO \eqref{1.3} is defined
by
$$\mathbf{x}^{m+1}=V(\mathbf{x}^{(m)}), \ \ m=0,1,2,\ldots$$

A QSO $V$ defined by \eqref{1.3} is called \textit{Volterra
operator} \cite{G} if one has
\begin{equation}\label{1.5}
P_{ij,k}=0 \ \ \mbox{if} \ \ k\not\in \{i,j\}, \ \ \forall i,j,k\in
E.
\end{equation}

Note that it is obvious that the biological behavior of condition
\eqref{1.5} is that the offspring repeats one of its parents'
genotype (see \cite{G,11}).

\begin{defn}
Let $\mathbf{x}=(x_1,\ldots,x_n)$ and $\mathbf{y}=(y_1,\ldots,y_n)$.
$\mathbf{x}$ is \textit{equivalent} to $\mathbf{y}$
($\mathbf{x}\sim\mathbf{y}$) if
\begin{itemize}
\item[(i)]$\mathbf{x}\prec\mathbf{y}$ ($\mathbf{x}$ is \textit{absolutely
continuous} with respect to $\mathbf{y}$) if $y_k=0\Rightarrow
x_k=0$, \item[(ii)]$\mathbf{y}\prec\mathbf{x}$ if $x_k=0\Rightarrow
y_k=0$.
\end{itemize}
\end{defn}

\begin{defn}
Let $I=\{1,2,\ldots,n\}$ and $Supp(x)=\{i\in I|x_i\neq0\}$. Then
$\mathbf{x}$ is \textit{singular} or \textit{orthogonal} to
$\mathbf{y}$ ($\mathbf{x}\bot\mathbf{y}$) if $Supp(\mathbf{x})\cap
Supp(\mathbf{y})=\emptyset$.
\end{defn}

Note that if $\mathbf{x}\bot\mathbf{y}\Rightarrow
\mathbf{x}\cdot\mathbf{y}=0$, whenever $x,y\in S^n$. Here
$\mathbf{x}\cdot\mathbf{y}$ stands for the usual scalar product in
$\br^n$.

\section{On Volterra QSO}

In this section we are going to give a characterization of Volterra
quadratic operator in terms of the above given order. Note that
dynamics of Volterra QSO was investigated in \cite{G}. Certain other
properties of such kind of operators has been studied in \cite{MS}.
Some generalizations of Volterra QSO were studied in
\cite{MS1,RN,RZ}.

Recall that the vertices of the simplex $S^{m-1}$ are described by
the elements $e_k=(\delta_{1k},\delta_{2k},\dots,\delta_{mk})$,
where $\delta_{ik}$ is the Kronecker's delta.

\begin{thm}
Let $V:S^{n-1}\rightarrow S^{n-1}$ be a QSO. Then the following
conditions are equivalent:

\begin{itemize}
\item[(i)] $V$ is a Volterra QSO;

\item[(ii)] one has $V(\mathbf{x})\prec \mathbf{x}$ for all $\mathbf{x}\in
S^{n-1}$.
\end{itemize}
\end{thm}

\begin{proof} (i)$\Rightarrow$ (ii). It is known \cite{G} that any
Volterra QSO can be represented as follows:
\begin{equation}\label{1v}
(V(x))_k=x_k\left(1+\sum\limits_{i=1}^{m}a_{ki}x_i\right), \
k=\overline{1,m},
\end{equation}
where $a_{ki}=-a_{ik}$, $|a_{ki}|\le 1$.

From the equality we immediately get $V(\mathbf{x})\prec \mathbf{x}$
for all $\mathbf{x}\in S^{n-1}$.

(ii)$\Rightarrow$ (i). Let $\mathbf{x}=e_k$, $(k\in\{1,\dots,n\})$.
Then due to $V(\mathbf{x})\prec \mathbf{x}$ from \eqref{1.3} one
finds
\begin{equation}\label{3v}
 P_{kk,k}=1 \qquad P_{kk,i}=0, \ i\neq k.
\end{equation}

Now assume that $\mathbf{x}=\l e_i+(1-\l)e_j$, where $\l\in(0,1)$.
Let $k\notin \{i,j\}$, then from \eqref{1.3} one finds that
\begin{equation}\label{2v}
V(\mathbf{x})_k=P_{ii,k}\l^2+2\l(1-\l)P_{ij,k}+P_{jj,k}(1-\l)^2
\end{equation}
Taking into account \eqref{3v} and the relation $V(\mathbf{x})\prec
\mathbf{x}$ with \eqref{2v} one gets $P_{ij,k}=0$. This completes
the proof.
\end{proof}

The proved theorem characterizes Volterra QSO in terms of absolute
continuity of distributions. Therefore, this theorem will allow to
define such kind of operators in abstract settings. Let us
demonstrate it.

Assume that $(E,\cf)$ be a measurable space and $S(E,\cf)$ be the
set of all probability measures on $(E,\cf)$.

Recall that a mapping $V :S(E,\cf)\to S(E,\cf)$ is called a
\textit{quadratic stochastic operator (QSO)} if, for an arbitrary
measure $\l\in S(E,\cf)$ the measure $\l'= V(\l)$ is defined as
follows
\begin{eqnarray}\label{VQ}
\l'(A)=\int_E\int_E P(x,y,A)d\l(x)d\l(y), \ \ A\in\cf,
\end{eqnarray}
where $P(x,y,A)$ satisfies the following conditions:
\begin{enumerate}
\item[(i)] $P(x,y,\cdot)\in S(E,\cf)$ for any fixed $x,y\in E$;

\item[(ii)] For any fixed $A\in\cf$ the function $P(x,y,A)$ is
measurable of two variables $x$ and $y$ on $(E\times
E,\cf\otimes\cf)$;

\item[(iii)] the function $P(x,y,A)$ is symmetric, i.e.
$P(x,y,A)=P(y,x,A)$ for any $x,y\in E$ and $A\in\cf$.
\end{enumerate}

Note that when $E$ is finite, i.e. $E=\{1,\dots,m\}$, then a QSO on
$S(E,\cf)= S^{m-1}$ is defined as in \eqref{1.3} with $P_{ij,k}=
P(i, j, k)$.

Certain construction of QSO in general setting was studied in
\cite{GR}.

We recall that a measure $\m\in S(E,\cf)$ is \textit{absolutely
continuous} w.r.t. a measure $\n\in S(E,\cf)$ if $\n(A)=0$ implies
$\m(A)=0$, and they are denoted by $\m\prec\n$. Put
$$
\textrm{null}(\m)=\bigcup\limits_{\m(A)=0} A.
$$
Then support of the measure $m$ is defined by
$supp(\m)=E\setminus\textrm{null}(\m)$. Two measures $\m,\n\in
S(E,\cf)$ are called \textit{singular} if $supp(\m)\cap
supp(\n)=\emptyset$, and they are denoted by $\m\perp\n$.

\begin{defn}
A QSO given by \eqref{VQ} is called \textit{Volterra} if
$V\l\prec\l$ for all $\l\in S(E,\cf)$.
\end{defn}

\begin{thm} Let $V$ be given by \eqref{VQ}. Then $V$ is Volterra QSO
if and only if $P(x,y,A)=0$ for all $x,y\notin A$.
\end{thm}

\begin{proof} First we assume that $V$ is Volterra QSO. Take any
$x,y\in E$ and consider the measure $\n=\frac{1}{2}(\d_x+\d_y)$,
where $\d_x$ is a delta-measure, i.e. $\d_x(A)=\chi_A(x)$. Then from
\eqref{VQ} one finds that
$$
V(\n)(A)=\frac{1}{4}\big(P(x,x,A)+P(y,y,A)+2P(x,y,A)\big).
$$
From $V\n\prec\n$ and $\n(A)=0$ (if $x,y\notin A$) we infer that
$V(\n)(A)=0$, this yields that $P(x,x,A)=P(y,y,A)=P(x,y,A)=0$ if
$x,y\notin A$.

Let us suppose that $P(x,y,A)=0$ is valid for all $x,y\notin A$.
Assume that for $\m\in S(E,\cf)$ one has $\m(B)=0$ for some
$B\in\cf$. Let us show that $V(\m)(B)=0$. Indeed, from \eqref{VQ}
and the conditions one gets
\begin{eqnarray*}
V(\m)(B)&=&\int_E\int_E P(x,y,B)d\m(x)d\m(y)\\[2mm]
&=&\int_{E\setminus B}\int_{E\setminus B} P(x,y,B)d\m(x)d\m(y)+
\int_{E\setminus B}\int_{B}
P(x,y,B)d\m(x)d\m(y)\\[2mm]
&&+\int_{B}\int_{E\setminus B} P(x,y,B)d\m(x)d\m(y)+\int_{B}\int_{B}
P(x,y,B)d\m(x)d\m(y)\\[2mm]
&=&\int_{E\setminus B}\int_{E\setminus B} P(x,y,B)d\m(x)d\m(y)=0
\end{eqnarray*}
This completes the proof.
\end{proof}

\section{Orthogonal Preserving(OP) QSO in 2D Simplex}

We recall that two vectors $\mathbf{x}$ and $\mathbf{y}$ belonging
to $S^{n-1}$ are called \textit{singular} or \textit{orthogonal} if
$\mathbf{x}\cdot\mathbf{y}=0$.

A mapping $V:S^{n-1}\to S^{n-1}$ is called is \textit{Orthogonal
Preserving (O.P.)} if one has $V(\mathbf{x})\perp V(\mathbf{y})$
whenever $\mathbf{x}\perp\mathbf{y}$.

In this section we are going to describe orthogonal preserving QSO
defined in 2D simplex.

Let us assume that $V:S^2\to S^2$ be a orthogonal preserving QSO.

This means that
\begin{equation*}
V(1,0,0)\perp V(0,1,0) \perp V(0,0,1)
\end{equation*}
Now from the definition of QSO, we immediately get
\begin{equation*}
(P_{11,1},P_{11,2},P_{11,3})\perp (P_{22,1},P_{22,2},P_{22,3})\perp
(P_{33,1},P_{33,2},P_{33,3})
\end{equation*}

Since in the simplex $S^2$ there are 3 orthogonal vectors which are
$$(1,0,0),\ \ (0,1,0), \ \ (0,0,1).$$ We conclude the vectors
$$(P_{11,1},P_{11,2},P_{11,3}), \ \ (P_{22,1},P_{22,2},P_{22,3}), \ \
(P_{33,1},P_{33,2},P_{33,3})$$ could be permutation of the given
orthogonal vectors. Therefore we have 6 possibilities and we
consider each of these possibilities one by one.

Let us first assume that
\begin{equation*}
\begin{matrix}
P_{11,1}=0 & P_{11,2}=0 & P_{11,3}=1 \\
P_{22,1}=0 & P_{22,2}=1 & P_{22,3}=0 \\
P_{33,1}=1 & P_{33,2}=0 & P_{33,3}=0
\end{matrix}
\end{equation*}
Now our aim is to find condition for the others coefficients of the
given QSO.  Let us consider the following vectors
\begin{equation*}
\mathbf{x}=\left(\frac{1}{2},\frac{1}{2},0 \right) \qquad
\mathbf{y}=(0,0,1)
\end{equation*}
which are clearly orthogonal. One can see that
\begin{align*}
V(\mathbf{x})&=1/4(2P_{12,1},2P_{12,2}+1,1+2P_{12,3}) \\
V(\mathbf{y})&=(1,0,0)
\end{align*}
Therefore, the orthogonal preservably of $V$ yields $P_{12,1}=0$.
From $\sum^3_{i=1}P_{12,i}=1$ one gets
\begin{equation*}
P_{12,2}+P_{12,3}=1
\end{equation*}
Now consider
\begin{equation*}
\mathbf{x}=\left(0,\frac{1}{2},\frac{1}{2} \right) \qquad
\mathbf{y}=(1,0,0)
\end{equation*}
Then we have
\begin{align*}
V(\mathbf{x})&=1/4(2P_{23,1}+1,1+2P_{23,2},1+2P_{23,3}), \\
V(\mathbf{y})&=(0,0,1)
\end{align*}
Again the orthogonal preservability of $V$ implies $P_{23,3}=0$ and
hence we get
\begin{equation*}
P_{23,1}+P_{23,2}=1
\end{equation*}

Now consider
\begin{equation*}
\mathbf{x}=\left(\frac{1}{2},0,\frac{1}{2} \right) \qquad
\mathbf{y}=(0,1,0)
\end{equation*}
Then one has
\begin{align*}
V(\mathbf{x})&=1/4(1+2P_{13,1},2P_{13,2},1+2P_{13,3}), \\
V(\mathbf{y})&=(0,1,0)
\end{align*}
Hence, we conclude that $P_{13,2}=0$ and get
\begin{equation*}
P_{13,1}+P_{13,3}=1
\end{equation*}
Taking into account the obtained equations, we denote
\begin{equation*}
P_{12,2}=\alpha \qquad P_{23,1}=\beta \qquad P_{13,1}=\gamma
\end{equation*}
Correspondingly one gets
\begin{equation*}
P_{12,3}=1-\alpha \qquad P_{23,2}=1-\beta \qquad P_{13,3}=1-\gamma
\end{equation*}
Therefore $V$ has the following form
\begin{equation*}
V^{(1)}_{\alpha,\beta,\gamma}: \left\{
\begin{array}{l}
x'=z^2+2\gamma xz+2\beta yz\\
y'=y^2+2\alpha xy+2(1-\beta)yz\\
z'=x^2+2(1-\alpha) xy+2(1-\gamma)xz\\
\end{array} \right.
\end{equation*}

Similarly, considering other possibilities we obtain the following
operators:

\begin{equation*}
V^{(2)}_{\alpha,\beta,\gamma} : \left\{
\begin{array}{l}
x'=x^2+2\alpha xy+2\gamma xz\\
y'=y^2+2(1-\alpha) xy+2\beta yz\\
z'=z^2+2(1-\gamma) xz+2(1-\beta)yz\\
\end{array} \right.
\end{equation*}

\begin{equation*}
V^{(3)}_{\alpha,\beta,\gamma}: \left\{
\begin{array}{l}
x'=x^2+2\alpha xy+2\gamma xz\\
y'=z^2+2(1-\gamma) xz+2\beta yz\\
z'=y^2+2(1-\alpha) xy+2(1-\beta)yz\\
\end{array} \right.
\end{equation*}

\begin{equation*}
V^{(4)}_{\alpha,\beta,\gamma}: \left\{
\begin{array}{l}
x'=y^2+2\alpha xy+2\beta yz\\
y'=z^2+2\gamma xz+2(1-\beta)yz\\
z'=x^2+2(1-\alpha) xy+2(1-\gamma)xz\\
\end{array} \right.
\end{equation*}

\begin{equation*}
V^{(5)}_{\alpha,\beta,\gamma}: \left\{
\begin{array}{l}
x'=y^2+2\alpha xy+2\beta yz\\
y'=x^2+2(1-\alpha) xy+2\gamma xz\\
z'=z^2+2(1-\gamma) xz+2(1-\beta)yz\\
\end{array} \right.
\end{equation*}

\begin{equation*}
V^{(6)}_{\alpha,\beta,\gamma}:\left\{
\begin{array}{l}
x'=z^2+2\gamma xz+2\beta yz\\
y'=x^2+2\alpha xy+2(1-\gamma)xz\\
z'=y^2+2(1-\alpha) xy+2(1-\beta)yz\\
\end{array} \right.
\end{equation*}

So, if $V$ is OP QSO, then it can be one of the above given
operators. Now we are going to show the obtained operators are
indeed orthogonal preserving.

\begin{thm}\label{OP} Let $V$ be an orthogonal preserving QSO. Then
$V$ has one of the following forms:
\begin{equation}\label{list}
V^{(1)}_{\alpha,\beta,\gamma}, \ V^{(2)}_{\alpha,\beta,\gamma}, \
V^{(3)}_{\alpha,\beta,\gamma}, \ V^{(4)}_{\alpha,\beta,\gamma}, \
V^{(5)}_{\alpha,\beta,\gamma}, V^{(6)}_{\alpha,\beta,\gamma}.
\end{equation}
\end{thm}

\begin{proof}
According to the above done calculation we have six listed
operators. Now we show that these operators indeed OP. Without loss
of generality, we may consider operator
$V^{(1)}_{\alpha,\beta,\gamma}$.

Assume that $\mathbf{x}\perp \mathbf{y}$. Then there are following
possibilities:
\begin{equation*}
\mathbf{x}\perp \mathbf{y}\Longleftrightarrow \left\{
\begin{array}{l}
\quad \mathbf{x}=(x,y,0) \quad \mathbf{y}=(0,0,1), \\
\quad \mathbf{x}=(x,0,z) \quad \mathbf{y}=(0,1,0), \\
\quad \mathbf{x}=(0,y,z) \quad \mathbf{y}=(1,0,0). \\
\end{array} \right.
\end{equation*}

Let $\mathbf{x}=(x,y,0)$ and $\mathbf{x}=(0,0,1)$. Then one gets
\begin{equation*}
V^{(1)}_{\alpha,\beta,\gamma}(\mathbf{x})=(0,y^2+2\alpha
xy,x^2+2(1-\alpha)xy),\qquad
V^{(1)}_{\alpha,\beta,\gamma}(\mathbf{y})=(1,0,0).
\end{equation*}
It is clear there are orthogonal. By the same argument, for other
two cases, we can establish the orthogonality of
$V^{(1)}_{\alpha,\beta,\gamma}(\mathbf{x})$ and
$V^{(1)}_{\alpha,\beta,\gamma}(\mathbf{y})$. This completes the
proof.
\end{proof}

\begin{rem}\label{permut} We note that the operators given in \eqref{list} are
permutations of Volterra QSO.  In \cite{GE} it was proved that
permutations of Volterra operators are automorphisms of the simplex.
\end{rem}

\begin{rem} It is well-known that linear stochastic operators are
orthogonal preserving if and only if they are permutations of the
simplex. We point out that if $\a=\b=\g=1/2$, then the operators
\eqref{list} reduce to such kind of permutations.
\end{rem}

 To investigate dynamic of obtained operators, it is usual to
investigate by means of the conjugacy.

Let us recall we say two QSO $V^{(1)}$ and $V^{(2)}$ are conjugate
if there exist a permutation
$T_{\pi}:(x,y,z)\rightarrow(\pi(x),\pi(y),\pi(z))$ such that
$T_{\pi}^{-1}V^{(1)}T_{\pi}=V^{(2)}$ and we denote  this by
$V^{(1)}\sim^\pi V^{(2)}$.

In our case, we need to consider only permutations of $(x,y,z)$
given by:
\begin{equation*}
\pi=
\begin{bmatrix}
x & y & z\\
y & z & x
\end{bmatrix} \qquad
\pi_1=
\begin{bmatrix}
x & y & z \\
x & z & y
\end{bmatrix}
\end{equation*}

Note that other permutations can be derived by the given two ones.

\begin{thm}\label{OP-K} Orthogonal
Preserving QSO can be divided into three non-conjugate classes
\begin{align*}
K_1&=\{V^{(1)}_{\alpha,\beta,\gamma},V^{(5)}_{\alpha,\beta,\gamma},V^{(3)}_{\alpha,\beta,\gamma}\} \\
K_2&=\{V^{(4)}_{\alpha,\beta,\gamma},V^{(6)}_{\alpha,\beta,\gamma}\} \\
K_3&=\{V^{(2)}_{\alpha,\beta,\gamma}\}
\end{align*}
\end{thm}

\begin{proof}
Let us consider $V^{(1)}_{\alpha,\beta,\gamma}$. Then one has
\begin{align*}
&T_{\pi}^{-1}V^{(1)}_{\alpha,\beta,\gamma}T_{\pi}(x,y,z)=T_{\pi}^{-1}V^{(1)}_{\alpha,\beta,\gamma}(y,z,x) \\
                                         &=(y^2+2(1-\alpha)yz+2(1-\gamma)yx,x^2+2\gamma yx+2\beta zx,z^2                                         +2\alpha yz+2(1-\beta)zx) \\
                                         &=V^{(5)}_{1-\gamma,1-\alpha,\beta}
\end{align*}
This means tht $V^{(1)}_{\alpha,\beta,\gamma}\sim^\pi
V^{(5)}_{1-\gamma,1-\alpha,\beta}$.

Similarly, we have
$T_{\pi}^{-1}V^{(5)}_{\alpha,\beta,\gamma}T_{\pi}(x,y,z)=V^{(3)}_{1-\gamma,\alpha,1-\beta}$.
Hence,$V^{(5)}_{\alpha,\beta,\gamma}\sim
V^{(3)}_{1-\gamma,\alpha,1-\beta}$. By the same argument one finds
$T_{\pi}^{-1}V^{(3)}_{\alpha,\beta,\gamma}T_{\pi}(x,y,z)=V^{(1)}_{\gamma,1-\alpha,1-\beta}$
which means $V^{(3)}_{\alpha,\beta,\gamma}\sim^\pi
V^{(1)}_{\gamma,1-\alpha,1-\beta}$.

This implies that the operators
$V^{(1)}_{\alpha,\beta,\gamma},V^{(5)}_{\alpha,\beta,\gamma},V^{(3)}_{\alpha,\beta,\gamma}$
are conjugate and we put them into one class denoted by $K_1$.

One can obtain that $V^{(2)}_{\alpha,\beta,\gamma}\sim^\pi
V^{(2)}_{1-\gamma,\alpha,1-\beta}$ and
$V^{(4)}_{\alpha,\beta,\gamma}\sim^\pi
V^{(4)}_{1-\gamma,1-\alpha,\beta}$ and
$V^{(6)}_{\alpha,\beta,\gamma}\sim^\pi
V^{(6)}_{\gamma,1-\alpha,1-\beta}$. Therefore we need to consider
another permutation $\pi_1$

Consequently, one finds $V^{(2)}_{\alpha,\beta,\gamma}\sim^{\pi_1}
V^{(2)}_{\gamma,1-\beta,\alpha}$,
$V^{(4)}_{\alpha,\beta,\gamma}\sim^{\pi_1}
V^{(6)}_{1-\gamma,\beta,\alpha}$.

Thus by $K_2$ we denote the class containing
$V^{(4)}_{\alpha,\beta,\gamma}$ and $V^{(6)}_{\alpha,\beta,\gamma}$
and by $K_3$ class containing only $V^{(2)}_{\alpha,\beta,\gamma}$.
This completes the proof.
\end{proof}

\begin{rem} One can see that the
operator $V^{(2)}_{\alpha,\beta,\gamma}$ is a Volterra QSO, and its
dynamics investigated in \cite{G}. From the result of \cite{V,Z} one
can conclude that even dynamics of Volterra QSO is very complicated.
We note that if $\a,\b,\g\in\{0,1\}$ then dynamics of operators
taken from the classes $K_1$,$K_2$ were investigated in
\cite{MJ,MSJ}. In \cite{GE} certain general properties of dynamics
of permuted Volterra QSO were studied.
\end{rem}

\begin{rem} We can also defined orthogonal preserving
in general setting. Namely, we call a QSO given by \eqref{VQ}
\textit{orthogonal preserving} if $V(\m)\perp V(\n)$ whenever
$\m\perp\n$, where $\m,\n\in S(E,\cf)$. Taking into account Remark
\ref{permut} we can formulate the following

\begin{conj} Let $V$ be a QSO given by \eqref{VQ}. Then $V$ is orthogonal preserving
if and only if there is a measurable automorphism $\a:E\to E$ (i.e.
$\a^{-1}({\cf})\subset\cf$) and a Volterra QSO $V_0$ such that
$V\m=V_0(\m\circ\a^{-1})$.
\end{conj}

\end{rem}

\section{Associativity of Orthogonality Preserving QSO}

In this section we state basic definitions and properties of
genetics algebras.

Let $V$ be a QSO and suppose that $\mathbf{x},\mathbf{y}\in
\mathbb{R}^n$ are arbitrary vectors, we introduce a multiplication
rule on $\mathbb{R}^n $ by
\begin{equation*}
\mathbf{x}\circ
\mathbf{y}=\frac{1}{4}\big(V(\mathbf{x}+\mathbf{y})-V(\mathbf{x}-\mathbf{y})\big)
\end{equation*}

This multiplication can be written as follows:
\begin{equation}\label{4.1}
(\mathbf{x}\circ \mathbf{y})_k=\sum^n_{i,j=1}P_{ij,k}x_iy_j
\end{equation}
where $\mathbf{x}=(x_1,\ldots ,x_n),\mathbf{y}=(y_1,\ldots ,y_n)\in
\mathbb{R}^n $.

The pair  $(\br^n,\circ)$ is called \textit{genetic algebra}. We
note the this algebra is commutative. This means
$\mathbf{x}\circ\mathbf{y}=\mathbf{y}\circ\mathbf{x}$. Certain
algebraic properties of such kind of algebras were investigated in
\cite{W-B,Ly2}. In general, the genetic algebra no need to be
associative. Therefore, we introduce the following

\begin{defn} A QSO $V$ is called \textit{associative} if the corresponding
multiplication given by \eqref{4.1} is associative, i.e
\begin{equation}\label{4.2}
(\mathbf{x}\circ\mathbf{y})\circ \mathbf{z}=\mathbf{x}\circ
(\mathbf{y}\circ\mathbf{z})
\end{equation} hold for all
$\mathbf{x},\mathbf{y},\mathbf{z}\in\br^n$.
\end{defn}

In this section we are going to find associative orthogonal
preserving QSO. According to the previous section, we have only
three classes of OP QSO. Now we are interested whether these
operators will be associative. Note that associativity of some
classes of QSO has been investigated in \cite{G2008}.

\begin{thm} The QSO  $V^{(2)}_{\alpha,\beta,\gamma}$ is associative
if and only if one of the following conditions are satisfied:
\begin{align*}
(1)\quad \alpha=0,\quad\beta=0,\quad \gamma=0 \\
(2)\quad \alpha=1,\quad\beta=0,\quad \gamma=0 \\
(3)\quad \alpha=0,\quad\beta=1,\quad \gamma=1 \\
(4)\quad \alpha=1,\quad\beta=1,\quad \gamma=1 \\
(5)\quad \alpha=1,\quad\beta=1,\quad \gamma=0 \\
(6)\quad \alpha=0,\quad\beta=1,\quad \gamma=0 \\
\end{align*}
\end{thm}

\begin{proof}

To show the associativity we will check the equality \eqref{4.2},
which can be rewritten as follows:
\begin{equation}\label{4.3}
\sum^3_{i,j=1}P_{ij,u}x_{i}\left(\sum^3_{m,k=1}P_{mk,j}y_{m}z_{k}\right)=\sum^3_{i,j=1}P_{ij,u}\left(\sum^3_{m,k=1}P_{mk,i}x_{m}y_{k}\right)z_j \qquad u=1,2,3 \\
\end{equation}
where we have use the following equalities
\begin{align*}
(x\circ y)\circ z&=\sum^3_{i,j=1}P_{ij,l}x_{i}\left(\sum^3_{m,k=1}P_{mk,j}y_{m}z_{k}\right) \\
x\circ (y\circ
z)&=\sum^3_{i,j=1}P_{ij,l}\left(\sum^3_{m,k=1}P_{mk,i}x_{m}y_{k}\right)z_j
\qquad l=1,2,3
\end{align*}
For $V^{(2)}_{\alpha,\beta,\gamma}$ the equality \eqref{4.3} can be
written as follows:

\begin{align*}
x_1(y_1z_1+\alpha y_1z_2+\gamma y_1z_3+\alpha y_2z_1+\gamma y_3z_1)      \\
+\alpha x_1((1-\alpha)y_1z_2+(1-\alpha)y_2z_1+y_2z_2+\beta y_2z_3+\beta y_3z_2) \\
+\gamma x_1((1-\gamma)y_1z_3+(1-\beta)y_2z_3+(1-\gamma)y_3z_1+(1-\beta)y_3z_2+y_3z_3) \\
+\alpha x_2(y_1z_1+\alpha y_1z_2+\gamma y_1z_3+\alpha y_2z_1+\gamma y_3z_1) \\
+\gamma x_3(y_1z_1+\alpha y_1z_2+\gamma y_1z_3+\alpha  y_2z_1+\gamma y_3z_1) \\
=z_1(x_1y_1+\alpha x_1y_2+\gamma x_1y_3+\alpha x_2y_1+\gamma x_3y_1) \\
+\alpha z_2(x_1y_1+\alpha  x_1y_2+\gamma x_1y_3+\alpha x_2y_1+\gamma x_3y_1) \\
+\gamma z_3(x_1y_1+\alpha x_1y_2+\gamma x_1y_3+\alpha x_2y_1+\gamma x_3y_1) \\
+\alpha z_1((1-\alpha)x_1y_2+(1-\alpha)x_2y_1+x_2y_2+\beta x_2y_3+\beta x_3y_2) \\
+\gamma
z_1((1-\gamma)x_1y_3+(1-\beta)x_2y_3+(1-\gamma)x_3y_1+(1-\beta)x_3y_2+x_3y_3);
\end{align*}
\begin{align*}
(1-\alpha)x_1((1-\alpha)y_1z_2+(1-\alpha)y_2z_1+y_2z_2+\beta y_2z_3+\beta y_3z_2) \\
+(1-\alpha)x_2(y_1z_1+\alpha y_1z_2+\gamma y_1z_3+\alpha y_2z_1+\gamma y_3z_1) \\
+x_2((1-\alpha)y_1z_2+(1-\alpha)y_2z_1+y_2z_2+\beta y_2z_3+\beta y_3z_2)  \\
+\beta x_2((1-\gamma)y_1z_3+(1-\beta)y_2z_3+(1-\gamma)y_3z_1+(1-\beta)y_3z_2+y_3z_3) \\
+\beta x_3((1-\alpha)y_1z_2+(1-\alpha)y_2z_1+y_2z_2+\beta y_2z_3+\beta y_3z_2) \\
=(1-\alpha)z_2(x_1y_1+\alpha x_1y_2+\gamma x_1y_3+\alpha x_2y_1+\gamma x_3y_1) \\
+(1-\alpha)z_1((1-\alpha)x_1y_2+(1-\alpha)x_2y_1+x_2y_2+\beta x_2y_3+\beta x_3y_2) \\
+z_2((1-\alpha)x_1y_2+(1-\alpha)x_2y_1+x_2y_2+\beta x_2y_3+\beta x_3y_2) \\
+\beta z_3((1-\alpha)x_1y_2+(1-\alpha)x_2y_1+x_2y_2+\beta x_2y_3+\beta x_3y_2) \\
+\beta
z_2((1-\gamma)x_1y_3+(1-\beta)x_2y_3+(1-\gamma)x_3y_1+(1-\beta)x_3y_2+x_3y_3);
\end{align*}
\begin{align*}
(1-\gamma)x_1((1-\gamma)y_1z_3+(1-\beta)y_2z_3+(1-\gamma)y_3z_1+(1-\beta)y_3z_2+y_3z_3) \\
+(1-\beta)x_2((1-\gamma)y_1z_3+(1-\beta y_2z_3+(1-\gamma)y_3z_1+(1-\beta)y_3z_2+y_3z_3) \\
+(1-\gamma)x_3(y_1z_1+\alpha y_1z_2+\gamma y_1z_3+\alpha y_2z_1+\gamma y_3z_1) \\
+(1-\beta)x_3((1-\alpha)y_1z_2+(1-\alpha)y_2z_1+y_2z_2+\beta y_2z_3+\beta y_3z_2) \\
+x_3((1-\gamma)y_1z_3+(1-\beta)y_2z_3+(1-\gamma)y_3z_1+(1-\beta)y_3z_2+y_3z_3) \\
=(1-\gamma)z_3(x_1y_1+\alpha x_1y_2+\gamma x_1y_3+\alpha x_2y_1+\gamma x_3y_1) \\
+(1-\beta)z_3((1-\alpha)x_1y_2+(1-\alpha)x_2y_1+x_2y_2+\beta x_2y_3+\beta x_3y_2) \\
+(1-\gamma)z_1((1-\gamma)x_1y_3+(1-\beta)x_2y_3+(1-\gamma)x_3y_1+(1-\beta)x_3y_2+x_3y_3) \\
+(1-\beta)z_2((1-\gamma)x_1y_3+(1-\beta)x_2y_3+(1-\gamma)x_3y_1+(1-\beta)x_3y_2+x_3y_3) \\
+z_3((1-\gamma)x_1y_3+(1-\beta)x_2y_3+(1-\gamma)x_3y_1+(1-\beta)x_3y_2+x_3y_3).
\end{align*}
Now equalizing the corresponding terms and simplifying the obtained
expressions one gets:
\begin{equation*}
\begin{matrix}
\beta(1-\beta)=0 & \alpha(1-\gamma)=(\alpha-\gamma)(1-\beta) & \alpha(1-\alpha)=0 \\
\alpha(\gamma-\beta)=0 & \gamma(1-\gamma)=0 & \gamma(1-\beta)=0\\
(\beta-\gamma)(1-\alpha)=\beta(1-\gamma)
\end{matrix}
\end{equation*}
Solving these equations we get the desired equalities which
completes the proof.
\end{proof}

By the same argument one can prove the following

\begin{thm} \label{Theorem:box} The operators $V^{(1)}_{\alpha,\beta,\gamma}$ and
$V^{(4)}_{\alpha,\beta,\gamma}$ are not associative for any values
of $\alpha,\beta,\gamma$.
\end{thm}

\section*{Acknowledgement} The author acknowledges the MOHE Grant ERGS13-024-0057. He also
thanks the Junior Associate scheme of the Abdus Salam International
Centre for Theoretical Physics, Trieste, Italy.


\begin{thebibliography}{99}

\bibitem{B} Bernstein S.N., The solution of a mathematical problem concerning the theory of heredity.
\emph{Ucheniye-Zapiski N.-I. Kaf. Ukr. Otd. Mat.} {\bf 1} (1924),
83-115 (Russian).


\bibitem{G2008} Ganikhodjaev N, Hisamuddin H.H.,  Associativity in inheritance or are there associative populations.
{\it Malaysian Journal of Science} {\bf 27(2)} (2008), 131--136.

\bibitem{GR} Ganikhodjaev N, Rozikov U., On quadratic stochastic operators generated by
Gibbs distributions, {\it Reg. Chaot. Dyn.} {\bf 11} (2006),
467--473.

\bibitem{G} Ganikhodzhaev R.N., Quadratic stochastic operators,
Lyapunov functions and tournaments. \emph{Russian Acad. Sci.
Sbornik. Math.} \textbf{76} (1993), 489-506.


\bibitem{GE} Ganikhodzhaev, R. N., Eshmamatova, D. B. Quadratic automorphisms
of a simplex and the asymptotic behavior of their trajectories. {\it
Vladikavkaz. Math. Jour.} {\bf 8} (2006), no. 2, 12--28.


\bibitem{11} Ganikhodzhaev R., Mukhamedov F., Rozikov U.,
Quadratic stochastic operators and processes: results and open
problems, \textit{Infin. Dimens. Anal. Quantum Probab. Relat. Top.}
{\bf 14}(2011) 270--335.

\bibitem{HHJ} Hofbauer J., Hutson V., Jansen W., Coexistence for systems
governed by difference equations of Lotka-–Volterra type,
\textit{J. Math. Biol.} \textbf{25} (1987) 553-–570.

\bibitem{HS} Hofbauer J., Sigmund K., Evolutionary Games and Population Dynamics,
\textit{Cambridge University Press}, Cambridge, 1998.

\bibitem{HS2} Hofbauer J., Sigmund K., The theory of evolution and dynamical
systems, \textit{Cambridge Univ. Press}, (1988).

\bibitem{L} Lotka A.J., Undamped oscillations derived from the law of mass action,
\textit{J. Amer. Chem. Soc.} {\bf 42} (1920), 1595--1599.

\bibitem{Ly2} Lyubich Yu.I., Mathematical structures in population genetics,
{\sl Biomathematics}, Springer-Verlag, {\bf 22} (1992).

\bibitem{May} May R.M., Simple mathematical models with very complicated
dynamics, \textit{Nature} {\bf 261} (1976) 459-–467

\bibitem{MO} May R.M., Oster G.F., Bifurcations and dynamic complexity in
simple ecological models, \textit{Am. Nat.} \textbf{110} (1976)
573–-599.

\bibitem{Moran} Moran P.A.P., Some remarks on animal population dynamics,
\textit{Biometrics} \textbf{6} (1950) 250–-258.

\bibitem{MS} Mukhamedov F., Saburov M., On homotopy of volterrian
quadratic stochastic operator, {\it Appl. Math. \& Inform. Sci.}
{\bf 4}(2010) 47--62.

\bibitem{MS1} Mukhamedov F., Saburov M., On Dynamics of Lotka-Volterra
type operators, \textit{Bull. Malay. Math. Sci. Soc.} {\bf
37}(2014), 59--64.

\bibitem{MJ} Mukhamedov F., Jamal A.H. M., On $\xi^s$-quadratic stochastic
operators in 2-dimensional simplex, In book:  \textit{Proc. the 6th
IMT-GT Conf. Math., Statistics and its Applications (ICMSA2010)},
Kuala Lumpur, 3-4 November 2010, Universiti Tunku Abdul Rahman,
Malaysia, 2010, pp. 159--172.

\bibitem{MSJ} Mukhamedov F., Saburov M., Jamal A.H.M., On dynamics of $\xi^s$-quadratic stochastic
operators, \textit{Inter. Jour. Modern Phys.: Conference Series}
{\bf 9} (2012), 299--307.

\bibitem{MSQ} Mukhamedov F.,  Saburov M., Qaralleh I. On
$\xi^{(s)}$-quadratic stochastic operators on two dimensional
simplex and their behavior, \textit{Abst. Appl. Anal.} {\bf 2013}
(2013), Article ID 942038, 12 p.

\bibitem{NSE} Narendra S.G., Samaresh C.M., Elliott W.M., On the
Volterra and other nonlinear moldes of interacting
populations,\textit{ Rev. Mod. Phys.} {\bf 43} (1971), 231--276.

\bibitem{PL} Plank M., Losert V.,
Hamiltonian structures for the n-dimensional Lotka-Volterra
equations, \textit{J. Math. Phys.} {\bf 36} (1995) 3520--3543.

\bibitem{RN} Rozikov U.A., Nazir S. Separable quadratic stochastic operators,
{\it Lobachevskii Jour. Math.} {\bf 31}(2010) 215–-221.

\bibitem{RZ} Rozikov U.A., Zada A. On $\ell$- Volterra Quadratic stochastic
operators. {\it Inter. Journal Biomath.} {\bf 3} (2010), 143--159.

\bibitem{16} Stein, P.R., Ulam S.M.,
{\it Non-linear transformation studies on electronic computers},
1962, Los Alamos Scientific Lab., N. Mex.


\bibitem{T} Takeuchi Y., Global dynamical properties of Lotka--Volterra systems,
\textit{World Scientific}, 1996.

\bibitem{UR} Udwadia F.E., Raju N., Some global properties of a pair of
coupled maps: quasi-symmetry, periodicity and syncronicity,
\textit{Physica D} \textbf{111} (1998) 16-–26.

\bibitem{U} Ulam S.M., A collection of mathematical problems.
{\it Interscience Publ. New York-London}, 1960.

\bibitem{V} Vallander S.S., On the limit behavior of iteration
sequence of certain quadratic transformations. \textit{Soviet Math.
Doklady}, {\bf 13}(1972), 123-126.

\bibitem{V1} Volterra V.,  Lois de fluctuation de la
population de plusieurs esp\`{e}ces coexistant dans le m\^{e}me
milieu,  \textit{Association Franc. Lyon} {\bf 1926} (1927), 96--98
(1926).

\bibitem{V2} Volterra V., Lecons sur la theorie mathematique de la lutte pour la vie,
{\it Gauthiers-Villars, Paris}, 1931.

\bibitem{W-B} Worz-Busekros, A., \textit{Algebras in Genetics},
Lecture Notes in Biomathematics, Vol. 36, Springer-Verlag, Berlin,
1980.
\bibitem{Z} Zakharevich M.I., On a limit behavior and ergodic hypothesis for quadratic mappings of a simplex.
{\it Russian Math. Surveys} {\bf 33} (1978), 207-208.

\end{thebibliography}
\end{document}